\documentclass[12pt,oneside,reqno]{amsart}
\newtheorem{theorem}{Theorem}[section]
\newtheorem{lemma}[theorem]{Lemma}
\newtheorem{proposition}[subsection]{\bf Proposition}

\theoremstyle{definition}

\newtheorem{example}[theorem]{Example}

\theoremstyle{remark}
\newtheorem{remark}[theorem]{Remark}

\usepackage{enumerate}
\usepackage{enumitem}
\usepackage{hyperref}
\newtheorem{corollary}{Corollary}[section]

\numberwithin{equation}{section}  
\def \C{\mathbb{C}}

\def \N{\mathbb{N}}

\def \Q{\mathbb{Q}}

\def \Z{\mathbb{Z}}

\def\house#1{\setbox1=\hbox{$\,#1\,$}%
\dimen1=\ht1 \advance\dimen1 by 2pt \dimen2=\dp1 \advance\dimen2 by 2pt
\setbox1=\hbox{\vrule height \dimen1 depth\dimen2\box1\vrule}%
\setbox1=\vbox{\hrule\box1}%
\advance\dimen1 by .4pt \ht1=\dimen1
\advance\dimen2 by .4pt \dp1=\dimen2 \box1\relax}
\begin{document}
	
\title{The quotient problem for linear recurrence sequences}
\author{Parvathi S Nair and S. S. Rout}%
\address[Parvathi S Nair]{Department of Mathematics, National Institute of Technology Calicut, 
Kozhikode-673 601, India.}
\email[]{parvathisnair60@gmail.com, parvathi\_p220245ma@nitc.ac.in}
\address[S. S. Rout]{Department of Mathematics, National Institute of Technology Calicut, 
Kozhikode-673 601, India.}
\email[]{sudhansu@nitc.ac.in}
\bigskip
\begin{abstract} 
Let $\{U(m)\}_{m\in \N}$ and $\{V(n)\}_{n\in \N}$ be linear recurrence sequences. It is a well-known Diophantine problem to determine the finiteness of the set of natural numbers $n$ such that the ratio $U(n)/V(n)$ is an integer. We study the finiteness problem for the set  $(m, n)\in \mathbb{N}^2$ such that there exist non-zero positive integers $d_{m, n}$ satisfying $\log |d_{m, n}|=o(n)$, and $d_{m, n}U(m)/V(n)$ is an element from a finitely generated subring of $\C$. In particular, we prove that for $m\neq n $, there exists a polynomial $P$ such that $d_{m, n}P(n)U(m)/V(n)$ is a multi-recurrence and $V(n)/P(n)$ is a linear recurrence and for $m=n$ both $d_{m, n}P(n)U(m)/V(n)$ and $V(n)/P(n)$ are linear recurrences. To prove our results, we employ Schmidt's subspace theorem, and the concept of moving hyperplanes, moving polynomials, and moving points.
\end{abstract}
\maketitle
	
\section{introduction}
A sequence of complex numbers $\{U(n)\}_{n\in \N}$ is called a linear recurrence of order $k$ if there exist complex numbers $c_0, \ldots, c_{k-1}$ such that for all $n \in \N$,
\[U(n+k)=c_{k-1}U(n+k-1)+\cdots+ c_0U(n),\]
with $k$ minimal. It is well known that for all $n\in \N$, any linear recurrence sequence can be explicitly written in the form
\begin{equation}\label{eq0}
U(n)=\sum_{i=1}^r u_i(n)\alpha_i^n,  
\end{equation}
with $u_i \in \C[X]$ and $\alpha_i \in \C^{\times}$, which are roots of the companion polynomial 
\[f(z)=z^k-c_{k-1}z^{k-1}-\cdots-c_0\]
of the recurrence sequence $\{U(n)\}$. We say $\{U(n)\}$ is simple if all the roots of $f(z)$ are simple and it is called non-degenerate if none of the ratio $\alpha_i/\alpha_j$ is a root of unity for  $i\neq j$. We refer the reader to (\cite{evertse}, \cite{shorey}) for the general theory of linear recurrences. 

\par
Many authors have worked on Diophantine equations involving recurrences (see, for example, \cite{zeros}, \cite{kthroot}, \cite{zannier}). The {\it Hadamard-quotient theorem} is a well-known result in the direction of linear recurrences. It states that \emph{if $U(n)$ and $V(n)$ are two linear recurrence sequences such that the ratio $\dfrac{U(n)}{V(n)}$ is an integer for all large $n \in \N$, then $\dfrac{U(n)}{V(n)}$ itself is a linear recurrence for all $n \in \N$}. In particular, it implies that given integers $a, b > 1$, if $a^n - 1$ divides $b^n - 1$ for all large positive integers $n$, then $b$ is a power of $a$. 

There is a close connection between the Hadamard-quotient theorem and the greatest common divisor (gcd) estimate for linear recurrences (see \cite{gcd, levin, Zhen}). In $2020$, Grieve and Wang \cite{gcd} obtained a moving form, in the sense of \cite{Ru1997}, of \cite[Theorem 1.11]{levin} and thereby derived consequences for the greatest common divisors of pairs of linear recurrence sequences defined over a number field.  Granville \cite{Granville2022} classified all linear division sequences in integers, which was a problem going back to the $1930$s. He also proved that if two linear division sequences have a large common factor infinitely often, then they are
each divisible by a common linear division sequence on some arithmetic progression. 

Note that Pisot's Hadamard quotient conjecture was solved by van der Poorten \cite{vdP2} and Pourchet \cite{Pourchet}. In fact, van der Poorten solved this problem in more general setting by assuming that the ratio of linear recurrences lies in a fixed finitely generated ring for all $n \in \N$. The proof by Pourchet-van der Poorten relies on an intricate auxiliary construction and on certain $p$-adic estimates. 
\par
 In \cite{corv1998}, Corvaja and Zannier considered simple recurrences with positive integer roots with coefficients from $\Q$. They proved that if the ratio between two simple recurrences with positive integer roots is an integer infinitely often, then the ratio itself is a linear recurrence of the same type.
\par

Later in 2002, Corvaja and Zannier \cite{corv} worked on the same problem for recurrences with polynomial coefficients. In \cite{corv}, they weakened the assumption to infinitely many $n$ in place of all $n$. In particular, they proved by assuming that \emph{$U(n), V(n)$} are linear recurrences such that their roots generate a torsion-free multiplicative group, and if $\dfrac{U(n)}{V(n)}$ lie in a finitely generated subring of $\C$ and $V(n)\neq 0$ for infinitely many $n\in \N$, then there exists a polynomial $P(x)\in \C[x]$ such that the sequences $P(n)U(n)/V(n)$ and $V(n)/P(n)$  are linear recurrences. The proof of the Corvaja-Zannier result uses Schmidt's subspace theorem.

Note that there is no simple converse of Theorem 2 in \cite{corv}. For example, if we take $U(n) = 5^n$ and $V(n) =n(n+1)$, then setting $P(n)= n(n+1)$ we infer that $P(n)U(n)/V(n)= 5^n$ and $V(n)/P(n)=1$ are both recurrence sequences. But all values $U(n)/V(n)$ do not belong to some fixed finitely generated ring, since the greatest prime factor of $n(n+1)$ tends to infinity with $n$. In this direction, the distribution of integral values for the ratio of two linear recurrences has also been studied by several authors (see \cite{divisibility},\cite{numbers n},\cite{sanna}).

In this paper, we consider the values $U(m)/V(n)$ are quasi-integral, in the sense that the denominators grow polynomially rather than exponentially, as suggested by the authors in  \cite[p.435]{corv}. To state our result, we need the following generalisation of the recurrence sequence. That is, \eqref{eq0} can be generalised by allowing more than one parameter as,
\begin{equation}\label{multi}
U(n_1,\ldots,n_t)=\sum_{i=1}^r u_i(n_1,\ldots,n_t)\alpha_{i1}^{n_1}\cdots\alpha_{it}^{n_t},  
\end{equation}
where $t,r$ are positive integers, $u_1,\ldots, u_r$ are polynomials in $t$ variables, and $n_1,\ldots,n_t$ are non-negative integers. The polynomial-exponential functions \eqref{multi} are called multi-recurrences. We say $U$ is defined over a field $F$ if the coefficients and bases $\alpha_{i1}, \ldots, \alpha_{it}$ are elements of $F$ for $i=1,\ldots, r$.

The main result of this paper is as follows.
\begin{theorem}\label{thm1}
Let $U(m), V(n)$ be linear recurrences defined over a number field $K$. Assume that their roots generate a torsion-free multiplicative group. Let $\mathcal R\subset K$ be a finitely generated subring.
\begin{enumerate}[label=\roman*)]
\item Assume that for infinitely many $n\in \mathbb{N}$, there exist non-zero positive integers $d_{n}$ such that $\log |d_{n}|=o(n)$, $V(n)\neq 0$, and $d_{n} U(n)/V(n) \in \mathcal R $. Then there exists a non-zero polynomial $P(X) \in \mathbb{C}[X]$ such that $ d_{n}P(n)U(n)/V(n)$ and $V(n)/P(n)$ are linear recurrences.
\item Assume that for all but finitely many $(m, n)\in \mathbb{N}^2$ with $m\neq n$, there exist non-zero positive integers $d_{m, n}$ such that $\log |d_{m, n}|=o(n)$, $V(n)\neq 0$, and $d_{m,n} U(m)/V(n) \in \mathcal R$. Then there exists a non-zero polynomial $P(X) \in \mathbb{C}[X]$ such that 
$d_{m, n}P(n)U(m)/V(n)$
is a multi-recurrence and $V(n)/P(n)$ is a linear recurrence.
\end{enumerate}
\end{theorem}

Note that the non-degeneracy of linear recurrences follows from the assumption that the roots form a torsion-free multiplicative group. Thus, $V(n)\neq 0$ for infinitely many $n\in \N$ can be described using Skolem-Mahler-Lech theorem (Theorem \ref{thmskm}).  

We present the following examples to illustrate the application of Theorem \ref{thm1}.
\begin{example}
Consider $K=\Q$, $U(m)=m2^m$, and $V(n)=n(1/3)^n$. Clearly, $U$ and $V$ are linear recurrences defined over $\Q$. The roots of $U$ and $V$ generate the torsion-free multiplicative group $\langle2,3\rangle$. If we choose $d_{m,n}=n$, we have $d_{m,n} U(m)/V(n)=m2^m3^n \in \Z$ for every tuple $(m, n)$. The polynomial $P(X)=X$ satisfies the conclusions of Theorem \ref{thm1}. That is,
$d_{m, n}P(n)U(m)/V(n)=nm2^m3^n$ is a multi-recurrence, and $V(n)/P(n)=(1/3)^n$ is a linear recurrence defined over $\Q$.
\end{example}

\begin{example}
Consider $K=\Q(\sqrt{5})$ and $\mathcal R=\Z\Big[\frac{1+\sqrt{5}}{2}\Big]$. We let  \[U(m)=\Bigg(\frac{1+\sqrt{5}}{2}\Bigg)^m \text{ and } V(n)=\Big(n+2\Big)\Big(\frac{1}{2}\Big)^n.\] Clearly, $U$ and $V$ are linear recurrences defined over $\Q(\sqrt{5})$. The roots of $U$ and $V$ generate the torsion-free multiplicative group $\Big\langle \frac{1+\sqrt{5}}{2},2\Big\rangle$. If we choose $d_{m,n}=n+2$
\[\frac{d_{m,n} U(m)}{V(n)}=\Bigg(\frac{1+\sqrt{5}}{2}\Bigg)^m2^n \in \mathcal R=\Z\Bigg[ \frac{1+\sqrt{5}}{2}\Bigg]\] 
for every tuple $(m, n)\in \mathbb{N}^2$. The polynomial $P(X)=1$ satisfies the conclusions of Theorem \ref{thm1}. That is,
\[\frac{d_{m, n}P(n)U(m)}{V(n)}=\Bigg(\frac{1+\sqrt{5}}{2}\Bigg)^m2^n\]
is a multi-recurrence sequence and $V(n)/P(n)=V(n)$ is a linear recurrence defined over $\Q(\sqrt{5})$.    
\end{example}

\begin{remark}
Part (ii) of Theorem \ref{thm1} does not hold if we replace the condition ``for all but finitely many $(m, n)$" with ``infinitely many $(m, n)$". This can be viewed by taking the counterexample 
\[U(m)=3^m-1, V(n)=2^n-1.\]
If $n=2k+1, k\in \Z_{\geq 0}$, $3$ is a unit modulo $2^n-1$, and $m=\phi(2^{n}-1)$, then the fraction $U(m)/V(n)$ is an integer, where $\phi$ is the Euler totient function. That is, we get an infinite set of tuples $(m, n)$ with $U(m)/V(n) \in \Z$. Here we may choose $d_{m, n}=1$. Hence, the assumptions of the theorem hold for an infinite set of tuples $(m, n)$, but we cannot find a non-zero polynomial $P(X)$ such that $P(n)(3^m-1)/(2^n-1)$ is a multi-recurrence. Also, when $n$ is even, there is no pair $(m, n)$ such that $U(m)/V(n) \in \Z$. So we get an infinite set of tuples $(m, n)$, which violates the assumptions of the theorem.
\end{remark}

\begin{remark}
The converse of Theorem \ref{thm1} is true. 
If the ratio of two linear recurrences defined over a number field $K$ is again a linear recurrence over $K$, then the ratio belongs to a finitely generated ring for all but finitely many $(m, n) \in \N^2$. 

In particular, if there exists a polynomial $P(X)$ such that $\frac{d_{m, n}P(n)U(m)}{V(n)}$ is a multi-recurrence and $V(n)/P(n)$ is a linear recurrence, where 
$\log |d_{m, n}|=o(n)$, then $d_{m, n}U(m)/V(n)$ is contained in a finitely generated ring $\mathcal{R}$. Note that any finitely generated ring can be taken as a subring of $\mathcal{O}_{K, S}$ by choosing the set $S$ of places of $K$ appropriately. To prove the converse of our result, we choose $S$ large enough so that the coefficients of $P(X), U(m), V(n)$ and the roots of $U(m), V(n)$ are $S$-units. Hence, we get $\dfrac{d_{m, n}P(n)U(m)}{V(n)}\in \mathcal{O}_{K, S}$. Since the coefficients of $P(X)$ are $S$-units, we can divide by $P(X)$ and get $d_{m, n}U(m)/V(n)\in \mathcal{O}_{K, S}$. Here, $\mathcal{O}_{K, S}$ is finitely generated by the coefficients and roots of recurrences. There is no change in the proof if $m=n$. Hence, the converse.
 \end{remark}

Using specialisation arguments, we can obtain the following more general result as a consequence of Theorem \ref{thm1}.

\begin{corollary}\label{cor1}
Let $U(m), V(n)$ be linear recurrences such that their roots together generate a torsion-free multiplicative group. Let $\mathcal R$ be a finitely generated subring of $\mathbb{C}$. 
\begin{enumerate}[label=\roman*)]
\item Assume that for infinitely many $n\in \mathbb{N}$, there exist non-zero positive integers $d_{n}$ such that $\log |d_{n}|=o(n)$, $V(n)\neq 0$, and $d_{n} U(n)/V(n) \in \mathcal R $. Then there exists a non-zero polynomial $P(X) \in \mathbb{C}[X]$ such that $ d_{n}P(n)U(n)/V(n)$ and $V(n)/P(n)$ are linear recurrences.
\item Assume that for all but finitely many $(m, n)\in \mathbb{N}^2$ with $m\neq n$, there exist non-zero positive integers $d_{m, n}$ such that $\log |d_{m, n}|=o(n)$, $V(n)\neq 0$, and $d_{m,n} U(m)/V(n) \in \mathcal R $. Then there exists a non-zero polynomial $P(X) \in \mathbb{C}[X]$ such that 
$d_{m, n}P(n)U(m)/V(n)$
is a multi-recurrence and $V(n)/P(n)$ is a linear recurrence.
\end{enumerate}
\end{corollary}

\begin{remark}
For the special case where $m=n$, Corollary \ref{cor1} gives Theorem 2 of \cite{corv}.
\end{remark}

\par 
The following corollary can easily be deduced from Corollary \ref{cor1}. 
\begin{corollary}\label{cor2}
Let $U(m), V(n)$ be linear recurrences, and let $\mathcal R$ be a finitely generated subring of $\mathbb{C}$. 
\begin{enumerate}[label=\roman*)]
\item Assume that for infinitely many $n\in \mathbb{N}$, there exist non-zero positive integers $d_{n}$ such that $\log |d_{n}|=o(n)$, $V(n)\neq 0$, and $d_{n} U(n)/V(n) \in \mathcal R$. Then there exists a non-zero polynomial $P(X) \in \mathbb{C}[X]$ and positive integers $q, r$ such that $d_{qn+r}P(qn+r)U(qn+r)/V(qn+r)$ and $V(qn+r)/P(qn+r)$ are linear recurrences.  
\item Assume that for all but finitely many $(m, n)\in \mathbb{N}^2$ with $m\neq n$, there exist non-zero positive integers $d_{m, n}$ such that $\log |d_{m, n}|=o(n)$, $V(n)\neq 0$, and $d_{m,n} U(m)/V(n) \in \mathcal R$. Then there exists a non-zero polynomial $P(X) \in \mathbb{C}[X]$ and positive integers $q, r_1, r_2$ such that $d_{qm+r_1, qn+r_2}P(qn+r_2)U(qm+r_1)/V(qn+r_2)$
is a multi-recurrence and $V(qn+r_2)/P(qn+r_2)$ is a linear recurrence.
\end{enumerate}
\end{corollary}

The organisation of this paper is as follows. In the next section, we give a definition of the Weil height, moving functions, moving hyperplanes, and various results, including Schmidt's subspace theorem (with moving targets), Skolem-Mahler-Lech theorem, etc., which are needed for the proofs of the above theorem. In Section \ref{sec3}, we prove Theorem \ref{thm1} and present specialisation arguments for Corollary \ref{cor1}. The proof of our result is inspired by the methods of \cite{corv}.

\section{Preliminaries}\label{sec2}
\subsection{Heights and Schmidt's subspace theorem }
Let $K\subset \mathbb{C}$ be a number field. Let $M_K$ be the set of places of $K$, and $M_K^{\infty}$ be the set of all archimedean places of $K$ and $M_K^0 = M_K\setminus M_K^{\infty}$. Let $S$ be a finite set of places of $K$ containing $M_K^{\infty}$. The ring of $S$-integers is 
\[\mathcal{O}_{K, S} := \{u\in K: |u|_{\omega} \leq  1  ~~ \forall \omega \notin S\}\]
and the group of $S$-units is defined as
\[\mathcal{O}_{K, S}^{\times} := \{u\in K: |u|_{\omega} = 1  ~~ \forall \omega \notin S\}.\]

For $x\in K^{\times}$, we put 
\begin{align*}
|x|_{\omega}:= 
\begin{cases}
|\sigma(x)|^2 & \text{if } \omega \in M_K \text{ is a complex place};  \\
|\sigma(x)| & \text{if } \omega \in M_K \text{ is a real place; and } \\
(\mathcal{O}_{K}:\mathfrak{p})^{ord_{\mathfrak{p}}(x)} & \text{if } \omega \text{ corresponds to a prime ideal } \mathfrak{p} \\& \text{ in the ring of integers } \mathcal{O}_{K}.
\end{cases}    
\end{align*}
Here, $\sigma$ is said to be a real embedding if $\omega$ is a real place and a complex embedding if $\omega$ is a complex place. With this notation, the product formula is valid for all non-zero $x$ in $K$.

Further, the absolute logarithmic Weil height $h(x)$ is defined as
$$
h(x):=\sum_{\omega\in M_K}\log \mbox{ max}\{1,|x|_\omega\} \mbox{ for all } x\in K.
$$
Note that this height is independent of the choice of the number field $K$ containing $x$. For a vector $\textbf{x} =(x_1,\ldots,x_n)\in K^n$ and for a place $\omega\in M_K$, the $\omega$-norm for $\textbf{x}$, denoted by $\|\textbf{x}\|_\omega$, is defined by 
$$
\Vert\textbf{x}\Vert_\omega:=\mbox{max}\{|x_1|_\omega,\ldots,|x_n|_\omega\}
$$
and the logarithmic height, $h(\textbf{x})$, is defined by 
$$
h(\textbf{x}):=\sum_{\omega\in M_K}\log \Vert\textbf{x}\Vert_\omega.
$$
	
Now we are ready to state a more general version of Schmidt's subspace theorem, which was formulated by Schlickewei \cite[Theorem 1D]{Subspace}.
	
\begin{theorem}[Subspace Theorem] \label{schli}
Let $K$ be a number field and $N \geq 1$ be an integer. Let $S$ be a finite set of places of $K$, which contains all the archimedean places of $K$. For each $\omega \in S$, let $L_{\omega,1}, \ldots, L_{\omega,N}$ be linearly independent linear forms in variables $X_1,\ldots,X_N$ with coefficients in $K$. For any $\varepsilon>0$, the set of solutions $\mathbf{x} \in \mathcal{O}_{K, S}^N$ to the inequality 
\begin{equation*}
\log  \Big(\prod_{\omega\in S}\prod_{i=1}^N {|L_{\omega, i}(\mathbf{x})|_{\omega}}\Big)< -\varepsilon h({\bf x)}
\end{equation*}
lies in a finite union of hyperplanes of $K^N$ defined over $K$.
\end{theorem}

We also need some additional definitions to prove our results. The height of a polynomial 
\[f(x_1,\ldots,x_n)=\sum_{{\bf i}=(i_1,\ldots, i_n)}a_{{\bf i}}x_1^{i_1}\cdots x_n^{i_n} \in K[x_1,\ldots,x_n]\]
is defined as
\[h(f):=\sum_{\omega\in M_K}\log \max_{{\bf i}}\{|a_{\bf i}|_{\omega}\}.\]
We set
\[\Vert{f}\Vert_\omega:= \max_{{\bf i}}\{|a_{\bf i}|_{\omega}\},\]  for $\omega \in M_K$. The
Weil function for the hyperplane $H \subseteq \mathbb{P}^n(K)$ defined by the linear form 
\[L({\bf x})=a_0x_0+\cdots+a_nx_n\] is given by
\[\lambda_{H,\omega}({\bf x}):=\log \Big (\frac{\Vert\textbf{x} \Vert_\omega \Vert{L}\Vert_\omega}{|a_0x_0+\cdots+a_nx_n|_{\omega}}\Big)\] for ${\bf x}=[x_0:\cdots: x_n]\in \mathbb{P}^n(K)\backslash H$ and $\omega \in M_K$.

\subsection{Moving functions and moving hyperplanes}
We adopt the following setup from \cite{Ru1997}.
Let $\Lambda$ be an infinite index set and $A\subseteq \Lambda$ be a fixed infinite subset. A moving hyperplane indexed by $\Lambda$ over $K$ is a map $H:\Lambda \rightarrow \mathbb{P}^n(K)^{\times},$ defined by $\alpha \rightarrow H(\alpha)$.

Let $H_j(\alpha), j=1,\ldots, q, \alpha \in \Lambda$ be the moving hyperplanes in $\mathbb{P}^n(K)$. Choose $a_{j,0}(\alpha),\ldots, a_{j, n}(\alpha) \in K$ such that $H_j(\alpha)$ is the hyperplane determined by the equation \[a_{j,0}(\alpha)x_0+\cdots+a_{j,n}(\alpha)x_n=0.\] 
An infinite subset $A \subseteq \Lambda$ is said to be coherent with respect to the collection $H_j(\alpha), j=1,\ldots, q$ of moving hyperplanes if, for every polynomial
\[P({\bf x})\in K[x_{1,0},\ldots, x_{1,n}, x_{2,0},\ldots,  x_{q,n}],\]
which is homogeneous in $x_{j,0},\ldots, x_{j,n}$ with $j=1,\ldots, q$, either 
\[P(a_{1,0}(\alpha),\ldots, a_{q, n}(\alpha))=0\] for all $\alpha\in A$; or it vanishes for at most finitely many $\alpha \in A$.

More generally, we define a collection of moving polynomials, $f_i(\alpha)$ indexed by $\Lambda$, for $i=1,\ldots, q$, of degree $d_i$ as 
\[f_i(\alpha)=\sum_{{\bf i}\in \mathcal{I}_{d_i}}a_{i,{\bf i}}(\alpha)x_1^{i_1}\cdots x_{n_i}^{i_{n_i}} \in K[x_1,\ldots,x_{n_i}],\]
where $\mathcal{I}_{d_i}$ is the set containing all monomials in $x_1,\ldots, x_{n_i}$ of degree less than or equal to $d_i$.

We also need the concept of moving points that are non-degenerate with respect to a collection of moving hyperplanes. For $i=1,\ldots, n$, consider a collection of maps $x_i:\Lambda \rightarrow K$ such that for all $\alpha \in \Lambda$, at least one $x_i(\alpha)\neq 0$. Such maps define moving points 
\begin{equation}\label{eqpoint}
{\bf x}(\alpha)=[x_0(\alpha):\cdots :x_n(\alpha)] \in \mathbb{P}^n(K)    
\end{equation} for each $\alpha \in \Lambda$. The moving points \eqref{eqpoint} are said to be non-degenerate with respect to a finite collection of moving hyperplanes $H_j(\alpha), j=1,\ldots, q$ if, for all infinite coherent subsets $A \subseteq \Lambda$, the restrictions of $x_0,\ldots, x_n$ to $A$ are linearly independent over the field of moving functions for the collection $H_j(\alpha), j=1,\ldots, q$ with respect to $A$.

The following theorem is Schmidt's subspace theorem with moving targets obtained by Ru-Vojta \cite[Theorem 1.1]{Ru1997}. This theorem plays a key role in the proofs of Lemma \ref{lemmagcd} and Lemma \ref{V(n)}. 
\begin{theorem}[Moving form of Schmidt's subspace theorem]\label{submoving}
Let $K$ be a number field, let $S$ be a finite set of places of $K$, let $q\in \Z_{>0}$, and let $\epsilon>0$. Let $\Lambda$ be an infinite index set, let $H=\{H_1,\ldots, H_q\}$ be a collection of moving hyperplanes in $\mathbb{P}^n$, and let ${\bf x}:\Lambda \rightarrow \mathbb{P}^n(K)$ be a collection of moving points such that 
\begin{enumerate}[label=\roman*)]
\item ${\bf x}:\Lambda \rightarrow \mathbb{P}^n(K)$ is non-degenerate with respect to $H$; and 
\item $h(H_j(\alpha))=o(h({\bf x}(\alpha)))$ for all $\alpha \in \Lambda$ and all $j=1,\ldots, q$.
\end{enumerate}
Then there exists an infinite index subset $A\subseteq \Lambda$ such that the inequality 
\[\sum_{\nu \in S}\max_{J}\sum_{j \in J}\lambda_{H_{j}(\alpha),\nu}({\bf x}(\alpha))\leq (n+1+\epsilon) h({\bf x}(\alpha))\]
holds for all $\alpha \in A$. Here, the maximum is taken over all subsets $J\subseteq \{1,\ldots, q\}$ such that the $H_j(\alpha), j\in J$ and all $\alpha \in \Lambda$ are linearly independent.
\end{theorem}

We need the next lemma to prove our result for the case $m\neq n$.
\begin{lemma}\cite[Lemma 3.4]{gcd}.\label{lemmagcd}
Let $K$ be a number field, $S$ be a finite set of places of $K$ containing the archimedean ones. Let ${\bf u}=[u_0:\cdots:u_n]$, where $u_0,\ldots, u_n: \Lambda \rightarrow  \mathcal{O}_{K, S}^{\times}$ is a sequence of maps. Let $H_{\alpha} \subset \mathbb{P}^n$ be a collection of moving hyperplanes defined by the linear forms $L_{\alpha}(x)\in K[x_0,\ldots, x_n]$ and with coefficients indexed by $\alpha \in \Lambda$. Assume that $h(L_{\alpha})=o(h({\bf u(\alpha)}))$ for all $\alpha \in \Lambda$. Let $\epsilon >0$. Then either
\begin{enumerate}[label=\roman*)]
\item there exists an infinite index subset $A$ of $\Lambda$ such that
\[\sum_{\nu \in S}\lambda_{H_{\alpha},\nu}({\bf u}(\alpha))<\epsilon h(({\bf u}(\alpha)))\] for all $\alpha \in A$; or
\item there exists an infinite index subset $A$ of $\Lambda$ and indices $i, j$ with $0\leq i\neq j \leq n$, such that
\[h(u_i(\alpha)/u_j(\alpha))=o(h({\bf u}(\alpha)))\] for all $\alpha \in A$.
\end{enumerate}
\end{lemma}

For a positive real number $a$, we denote $\log^-(a):=\min \{0, \log (a)\}$. 
\begin{lemma}\cite[Lemma 4.1]{gcd}\label{V(n)}
Let $V(n)=\sum_{i=1}^r v_i(n)\beta_i^n$ be a non-degenerate algebraic linear recurrence sequence with roots and coefficients over a number field $K$. Let $\nu\in M_K$ be such that $|\beta_i|_{\nu}\geq 1$ for some $i$, and let $\epsilon>0$. Then 
\begin{equation}\label{bound}
 -\log^{-}|V(n)|_{\nu}<\epsilon n
\end{equation} for all but finitely many $n\in \N$.
\end{lemma}

\section{Proof of Main Results}\label{sec3}
\subsection{Some results on recurrence sequences}
The zero set of a linear recurrence $\{U(n)\}$ is defined as 
\[\mathcal{Z}(U)= \{n\in \N: U(n)=0\}.\]
The Skolem-Mahler-Lech theorem can be stated as follows:
\begin{theorem}[Skolem-Mahler-Lech]\label{thmskm}
Let $\{U(n)\}$ be a linear recurrence sequence defined over a field of characteristic $0$. Then the set $\mathcal{Z}(U)$ is a union of a finite number of arithmetic progressions and a finite set.     
\end{theorem}
This result was proved for linear recurrences over the rationals by Skolem \cite{skolem}. It was extended to linear recurrences over algebraic numbers by Mahler \cite{Mahler1} and later generalised to $\C$ by Lech \cite{lech} (see also \cite{Mahler3}, \cite{Mahler2}).
We also need the following lemma to prove our result (see \cite{Rumely}, \cite{VdP}). This says that the divisibility properties of recurrence sequences, such as co-primality, can be understood in the same sense as in the Laurent ring $\C[X, T_1, \ldots, T_t, T_1^{-1}, \ldots, T_t^{-1}]$.

\begin{lemma}\cite[Lemma 2.1]{corv}\label{lem2.2}
Let $\Gamma \subset \C^{\times}$ be a torsion-free multiplicative subgroup of rank $t\geq 1$. The ring of linear recurrences $R_{\Gamma}$, whose roots belong to $\Gamma$, is isomorphic to the ring $\C[X, T_1, \ldots, T_t, T_1^{-1}, \ldots, T_t^{-1}]$. In particular, it is a unique factorisation domain.
\end{lemma}	

\subsection{A key result}
 The following proposition is an important result to prove Theorem \ref{thm1}. 

\begin{proposition}\label{prop}
Let $K$ be a number field, and let $S$ be a finite set of places of $K$ that contains the archimedean ones. Let $U(m), V(n)$ be linear recurrences with roots and coefficients in $K$. Suppose that the roots of $U$ and $V$ generate a torsion-free multiplicative subgroup $\Gamma$ of $K^{\times}$. Suppose that there exist non-zero integers $d_{m, n}$ with $\log |d_{m, n}|=o(n)$ such that $d_{m, n}U$ and $V$ are co-prime with respect to $\Gamma$. Also, assume that $V$ has more than one root. Then the following assertions are true:
\begin{enumerate}[label=\roman*)]
\item  There exist only finitely many $n\in \N$ in the case $m=n$ such that we can find non-zero integers $d_{n}$ satisfying $d_{n}U(n)/V(n) \in \mathcal{O}_{K, S}$.
\item There do not exist infinitely many pairs of natural numbers $(m, n)$ in the case $m\neq n$ with the property that $m=o(n)$, and there exist non-zero integers $d_{m, n}$ satisfying $d_{m, n}U(m)/V(n) \in \mathcal{O}_{K, S}$.
\end{enumerate}
\end{proposition}
\begin{proof}
Consider the case $m=n$. Without loss of generality, we may choose $S$ sufficiently large so that it is a finite set of places of $K$ containing all the archimedean places, and such that all the roots and non-zero coefficients of $U$ and $V$ are $S$-units in $K$. By hypothesis, roots of $U$ and $V$ generate a torsion-free multiplicative subgroup $\Gamma$ of $K^{\times}$, and $V$ has more than one root. 

Suppose $\beta_i$ and $\beta_j$ are roots of $V$, and $\beta_i/\beta_j$ is a root of unity, i.e., $(\beta_i/\beta_j)^{k} =1$ for some integer $k$. Since $\beta_i/\beta_j\in \Gamma$, which is torsion-free, we can conclude that the ratio of any two roots of $V$ cannot be a root of unity. If the $\nu$-adic absolute value for any two roots $\beta_i$ and $\beta_j$ of $V$ is the same for all $\nu\in M_K$, then we deduce that the absolute logarithmic Weil height $h(\beta_i/\beta_j)=0$, that is, $\beta_i/\beta_j$ is a root of unity, which is a contradiction. So, we can find a place $v\in M_K$ such that not all the roots of $V$ have the same $v$-adic absolute value. Clearly, $v \in S$. Let $\beta$ be the root of $V$ with the maximum $v$-adic absolute value. For simplicity, we replace $U(n)$ with $U(n)/\beta^n$ and $V(n)$ with $V(n)/\beta^n$. This will result in the case when the maximal $v$-adic absolute value of the roots of $V(n)$ is $1$. Using this, we can write $V(n)$ as 
$$V(n)=V_1(n)-W(n),$$, where all the roots of $V_1(n)$ have $v$-adic absolute value $1$ and all the roots of $W(n)$ have $v$-adic absolute value strictly less than $1$. Thus, we can find a positive real number $\delta<1$ such that 
\begin{equation}\label{eq2}
|W(n)|_v < c_1\delta^n   
\end{equation}
for some constant $c_1$.

Consider the subgroup $\Gamma^*$ of the free abelian multiplicative group $\Gamma$ formed with elements of $v$-adic absolute value $1$. Since $\Gamma/\Gamma^*$ is torsion-free, $\Gamma^*$ is a primitive subgroup. Then there exists a basis $\beta_1,\ldots,\beta_t$ for $\Gamma$ such that $\beta_1,\ldots,\beta_p$ forms a basis for $\Gamma^*$.
\par
Since the roots of $V_1$ lie in $\Gamma^*$, we may write 
\begin{equation}\label{eq11}
 V_1(n)=f(n,\beta_1^n,\ldots, \beta_p^n),   
\end{equation}
where $f\in K[X, T_1, T_1^{-1},\ldots, T_p, T_p^{-1}]$. We multiply both $U$ and $V$ by a suitable power of $\beta_1^n\cdots \beta_p^n$, and we may assume that $f$ is a polynomial of total degree, say $D$, in its arguments.
\par 

For $V(n) \neq 0$, set $e_n:=\frac{d_nU(n)}{V(n)}$. On the contrary, assume that there exists an infinite set $\mathcal{M}$ with 
\[\mathcal{M}=\{n\in \N: V(n) \neq 0, e_n \in \mathcal{O}_{K, S}\}.\] 
\par 
First, we will estimate $V_1(n)^q$ for a fixed positive integer $q$, that is,  
\begin{equation*}
V_1(n)^q= (V(n)+W(n))^q=V(n)\left(\sum_{i=0}^{q-1}\binom{q}{i}V(n)^{q-1-i}W(n)^i\right)+W(n)^q.
\end{equation*}
Thus,
\begin{align*}
V_1(n)^qe_n&=V_1(n)^q\frac{d_nU(n)}{V(n)}\\&=U(n)d_n\left(\sum_{i=0}^{q-1}\binom{q}{i}V(n)^{q-1-i}W(n)^i\right)+W(n)^q\frac{d_nU(n)}{V(n)}
\end{align*}
and this implies with \eqref{eq2}
\begin{align}\label{eq3}
\begin{split}
\Bigg|V_1(n)^qe_n-&U(n)d_n\left(\sum_{i=0}^{q-1}\binom{q}{i}V(n)^{q-1-i}W(n)^i\right)\Bigg|_v\\
&=|W(n)^qe_n|_v < c_2\delta^{nq}|e_n|_v.
\end{split}
\end{align}
Denote 
\[{\bf g}=(g_1,\ldots, g_p)\in \Z_{\geq 0}^{p}  \;\mbox{and}\;\;  {\bf b}=(b_1,\ldots, b_p)\in \Z_{\geq 0}^{p}.\]
Fix two other positive integers $h$ and $k$, which we will determine later. For every ${\bf g}$ with $g_1+\cdots+g_p \leq h,$ and every $z \in \Z_{\geq 0}$ with $z<k$, consider 
\begin{align}\label{eq 4}
\begin{split}
\Psi_{{\bf g},z}(n)&:= \left(V_1(n)^qe_n-U(n)d_n\sum_{i=0}^{q-1}\binom{q}{i}V(n)^{q-1-i}W(n)^i\right)n^z\underline{\gamma}^{n{\bf g}}\\&=\left(V_1(n)^qe_n+d_nH(n)\right)n^z\underline{\gamma}^{n{\bf g}},
\end{split}
\end{align}
where we set
$H(n):=-U(n)\sum_{i=0}^{q-1}\binom{q}{i}V(n)^{q-1-i}W(n)^i$ and $\underline{\gamma}^{(a_1,\ldots,a_p)}=\beta_1^{a_1}\cdots \beta_p^{a_p}$. Since $n$ is an integer, $|n|_v \leq n$.
Now using \eqref{eq3} and the fact that $|\beta_i|_v=1$ for $i=1,\ldots, p$, we get 
\begin{equation}\label{eq5}
 |\Psi_{{\bf g},z}(n)|_v <c_2 \delta^{nq}|e_n|_vn^z.  
\end{equation}
Recall from \eqref{eq11} that $V_1$ is a polynomial $f(n, \beta_1^n, \ldots, \beta_p^n)$ of degree less than or equal to $D$ in the variables $n, \beta_1^n, \ldots, \beta_p^n$. Hence, $f$ can be written in the form
\[f(n, \beta_1^n, \ldots, \beta_p^n)=\sum_{{\bf i}:=(i_0, \ldots, i_p)} a_{\bf i}n^{i_0}\beta_1^{ni_1}\cdots \beta_p^{ni_p}.\]
Then we write the first term on the right side of \eqref{eq 4} as 
\begin{align}\label{eq 7}
 \begin{split}
n^z\underline{\gamma}^{n{\bf g}}V_1(n)^qe_n&= (f(n, \beta_1^n, \ldots, \beta_p^n))^q n^z\beta_1^{ng_1}\cdots \beta_p^{ng_p}e_n\\&=\Bigg(\sum_{{\bf i}:=(i_0, \ldots, i_p)} a_{\bf i}n^{i_0}\beta_1^{ni_1}\cdots \beta_p^{ni_p}\Bigg)^q n^z\beta_1^{ng_1}\cdots \beta_p^{ng_p}e_n\\&= \sum_{{\bf b}, l}p_{{\bf b}, l, {\bf g}, z}n^l{\underline\gamma}^{n{\bf b}}e_n,
 \end{split}   
\end{align}
where the coefficients $p_{{\bf b}, l, {\bf g}, z} \in K$.
Observe that $l$ is the power of $n$ in the expression. Since $z<k$, $l$ will be at most $z+qD<k+qD$. The vector ${\bf b}$ is chosen such that the power of each $\beta_i$ is $nb_i$ for $i=1,\ldots, p$. Observe that $b_i$ is at most $g_i+qD$. 
Since $g_1+\cdots+g_p \leq h,$ then the index $({\bf b}, l)$ runs over the vectors $(b_1,\ldots, b_p,l)\in \Z_{\geq 0}^{p+1}$ with $b_1+\cdots+b_p \leq h+qD, 0\leq l < k+qD$.
Let 
\begin{equation}\label{M1}
M_1:=\binom{p+h+qD}{p}\cdot(k+qD). 
\end{equation}
Note that, $M_1$ represents the number of monomials of the form $X^lT_1^{b_1}\cdots\\ T_p^{b_p}$, with    $0\leq l<k+qD$ and $b_1+\cdots+b_p \leq h+qD$. The number of nonzero terms on the right side of \eqref{eq 7} is less than or equal to $M_1$. Next, we choose an ordering for the $M_1$ terms of the form,  $n^{l}{\underline\gamma}^{n{\bf b}}e_n$ with $0\leq l < k+qD$ and $b_1+\cdots+b_p\leq h+qD$. Denote $M_1$ tuples $({\bf b}, l)$ after reordering  by $({\bf b}_1, l_1), \ldots, ({\bf b}_{M_1}, l_{M_1})$.

The second term in the equation \eqref{eq 4} is $n^z\underline{\gamma}^{n{\bf g}}d_nH(n)$. Observe that $H(n)$ may be expressed as a sum of terms of the form $n^{\nu}\alpha^n$, for suitable $\nu$  and $\alpha \in \Gamma$. So, the second term on the right side of \eqref{eq 4} can be written as a linear combination of terms of the form $n^{l}{\underline\gamma}^{n{\bf g}}\alpha^nd_n$, for suitable ${\bf g}, l$ and  $\alpha$. In particular, it is of the form $n^{\nu}\alpha^nd_n$ for suitable quantities $\nu \in \N$ and $\alpha \in \Gamma$. Let $M_2$ denote the number of such terms. We also give an ordering for the $M_2$-mentioned terms. After ordering, let the $M_2$ tuples $({\bf g}, l)$ be $({\bf g}_{M_1+1}, l_{M_1+1}),\ldots, ({\bf g}_{N}, l_{N})$, where  $N=M_1+M_2$.
Hence $\Psi_{{\bf g},z}(n)$ is a linear combination of at most $N$ nonzero terms of the mentioned type.
Fix
\begin{align*}
&{\bf x}(n):=(x_1(n), \ldots, x_{N}(n))\\&=(n^{l_1}{\underline\gamma}^{n{\bf b}_1}e_n, \ldots, n^{l_{M_1}}{\underline\gamma}^{n{\bf b}_{M_1}}e_n, n^{l_{M_1+1}}{\underline\gamma}^{n{\bf g}_{M_1+1}}\alpha^nd_n, \ldots, n^{l_{N}}{\underline\gamma}^{n{\bf g}_{N}}\alpha^nd_n).
\end{align*}
Rewrite \eqref{eq 7} as
\[n^z\underline{\gamma}^{n{\bf g}}V_1(n)^qe_n= A_{{\bf g}, z, 1} x_1(n)+\cdots+A_{{\bf g}, z, M_1} x_{M_1}(n).\]
Note that the coefficients $A_{{\bf g}, z, i}$ for $i=1,\ldots, M_1$ are the same as the coefficients $p_{{\bf b}, l, {\bf g}, z}$ in \eqref{eq 7} in a suitable order.
Also,
\[n^z\underline{\gamma}^{n{\bf g}}d_nH(n)=A_{{\bf g}, z, M_1+1} x_{M_1+1}(n)+\cdots+A_{{\bf g}, z, N} x_{N}(n).\]
Since by our assumption $e_n \in \mathcal{O}_{K, S}$ for all $n \in \mathcal{M}$, the coordinates of the point ${\bf x}(n)$ are $S$-integers for all $n \in \mathcal{M}$.
Further, we define an ordering for the vectors $({\bf g}, z)\in \Z_{\geq 0}^p\times\Z_{\geq 0}$ with $g_1+\cdots+g_p \leq h$ and $0\leq z <k$. Let $M$ be the number of such tuples. Then 
\[M:=\binom{p+h}{p}\cdot k.\]
Since $q>0$ then, from \eqref{M1} it is clear that $M<M_1$. If $({\bf g}^j, z)$ is the $j$-th vector with respect to the chosen ordering, we put 
\begin{equation}\label{eq 8}
L_j(X_1,\ldots, X_N)=\sum_{i=1}^{N} A_{{\bf g}^j, z, i}X_i, \hspace{ 1cm}j=1,\ldots, M,    
\end{equation}
and we observe that $\Psi_{{\bf g}^j,z}(n)= L_j(x_1(n
),\ldots, x_N(n))$. We claim that the linear forms 
\[L_1(X_1, \ldots, X_{M_1}, 0, \ldots, 0), \ldots, L_M(X_1, \ldots, X_{M_1}, 0, \ldots, 0)\] are linearly independent.
\par
From \eqref{eq 8}, we have 
\[L_j(x_1(n), \ldots, x_{M_1}(n), 0, \ldots, 0)=n^z\underline{\gamma}^{n{\bf g}^j}V_1(n)^qe_n.\] Suppose that the linear forms are linearly dependent, then there exists a linear relation for all $n \in \N$ of the form 
$$\sum_{{\bf g}^j, z}c_{{\bf g}^j, z}n^z\underline{\gamma}^{n{\bf g}^j}V_1(n)^qe_n=0,$$ with not all $c_{{\bf g}^j, z}$ equal to $0$. By Theorem \ref{thmskm}, we know that $V_1(n)^qe_n=0$ only for finitely many $n$. Since $\beta_i$ are multiplicatively independent, no ratio of two terms of the form $\underline{\gamma}^{n{\bf g}^j}$ can be a root of unity for two distinct values of ${\bf g}^j$. Thus, by Theorem \ref{thmskm}, $\sum_{{\bf g}^j, z}c_{{\bf g}^j, z}n^z\underline{\gamma}^{n{\bf g}^j}$ is nonzero for large $n$. This gives a contradiction. Hence, $L_1,\ldots, L_M$ are linearly independent.
\par
Renumbering the first $M_1$ terms, if necessary, and using the above claim, we may assume that $L_1,\ldots, L_M, X_{M+1},\ldots, X_N$ are linearly independent.
\par
Define the linear forms $L_{w, j}({\bf X}) \in K[X_1,\ldots, X_N]$ in $N$ variables as follows. For the fixed $v\in S$ and $1\leq j\leq M$, put 
\[L_{v,j}({\bf X})=L_j({\bf X})\]
 and for all other pairs $(w, j) \in S\times\{1,\ldots, N\}$, put
\[L_{{w},j}({\bf X})=X_j.\]
Observe that the linear forms defined above are linearly independent for each $w \in S$. We apply the subspace theorem (Theorem \ref{schli}) for this choice of linear forms. For that, consider 
\begin{equation}\label{eq 9}
\log  \Big(\prod_{w\in S}\prod_{j=1}^N {|L_{w, j}(x_1(n),\ldots,x_{N}(n))|_{w}}\Big).  
\end{equation}
We know that for $j\leq M_1$, $x_j(n)$ are of the form $n^l{\underline\gamma}^{n{\bf b}}e_n$ for a suitable vector $({\bf b}, l)$ depending on $j$. In this case, $x_j(n)=0$ if and only if $e_n$ vanishes. Then, according to Theorem \ref{thmskm}, this will happen only for finitely many $n$. We disregard this finite set, so we assume that  $x_j(n)\neq 0$ for $j=1,\ldots, M_1$. Rewrite \eqref{eq 9} as
\begin{equation}\label{eq 10}
\log \left(\left(\prod_{w\in S}\prod_{j=1}^N {|x_j(n)|_{w}}\right)\left(\prod_{j=1}^M \frac{|L_{v, j}(x_1(n),\ldots,x_{N}(n))|_{v}}{|x_j(n)|_{v}}\right)\right).  
\end{equation}
We apply the subspace theorem to the vectors $(x_1(n),\ldots,x_{N}(n))$, for $n \in \mathcal{M}$ in the statement of the proposition, and by our assumption, this $\mathcal{M}$ is infinite. Note that all the coordinates of these vectors are $ S$-integers. 
\par
We have that the terms $x_j(n)$ are either of the form $n^l{\underline\gamma}^{n{\bf b}}e_n$ (for $j\leq M_1$) or of the form  
$n^l\alpha^nd_n$ (for $M_1< j\leq N$) for suitable integer $l$ and $S$-units $\alpha$ depending on $j$. Let $L$ be an upper bound for all exponents $l$ of $n^l$ in these expressions.
\par 
Upon considering the product, the $S$-unit part vanishes by the product formula. Using the fact that $e_n \in \mathcal{O}_{K, S}$ and $d_n$ are integers, we have
\begin{align*}
&\log \left(\prod_{w\in S}\prod_{j=1}^N {|x_j(n)|_{w}}\right)\\
&=\log \left(\prod_{w\in S}\prod_{j=1}^{M_1} {|n^l{\underline\gamma}^{n{\bf b}}e_n|_{w}}\right)+\log \left(\prod_{w\in S}\prod_{j=M_1+1}^N {|n^l\alpha^nd_n|_{w}}\right)\\& \leq NL\log n+ \log \Bigg(\prod_{w\in S}\prod_{j=1}^{M_1}|e_n|_{w}\Bigg)+\log \Bigg(\prod_{w\in S}\prod_{j=M_1+1}^{N}|d_n|_{w}\Bigg)\\& \leq NL\log n+ M_1 h(e_n)+ (N-M_1)h(d_n),
\end{align*}
where $l$ and ${\bf b}$ depending on $j$. Since $\log |d_n|=o(n)$, then $h(d_n)/n \rightarrow 0$. Hence $h(d_n)< \varepsilon n$, for large values of $n$ and for some $\varepsilon>0$. The above inequality becomes
\begin{equation}\label{eq 11}
\log \left(\prod_{w\in S}\prod_{j=1}^N {|x_j(n)|_{w}}\right)\leq NL\log n+ M_1 h(e_n)+ (N-M_1)\varepsilon n.     
\end{equation}
Since $|{\underline\gamma}^{n{\bf b}}|_v=1$, then for $j=1,\ldots, M$, 
\begin{equation*}
\log |x_j(n)|_{v}= \log |e_n|_v+l \log|n|_v,    
\end{equation*}
for suitable $l \in \{0,\ldots, L\}$, depending on $j$. 
Next, consider the second term of \eqref{eq 10}
\begin{align}\label{eq12}
\begin{split}
&\log \left(\prod_{j=1}^M \frac{|L_{v, j}(x_1(n),\ldots,x_{N}(n))|_{v}}{|x_j(n)|_{v}}\right)\\ &= \sum_{j=1}^M \Bigg(\log (|L_{v, j}(x_1(n),\ldots,x_{N}(n))|_{v})-\log |x_j(n)|_{v} \Bigg)\\&=M \Bigg(\log |\Psi_{{\bf g}^j,z}(n)|_v -(\log |e_n|_v+l \log|n|_v)\Bigg)\\& \leq M(c_3nq\log \delta+2L\log n).
 \end{split}
 \end{align}
Finally, plugging \eqref{eq 11} and \eqref{eq12} in \eqref{eq 10} for large values of $n \in \mathcal{M}$, we get
\begin{align}\label{eq13}
\begin{split}
    &\log  \Big(\prod_{w\in S}\prod_{j=1}^N {|L_{w, j}(x_1(n),\ldots,x_{N}(n))|_{w}}\Big)\\&\leq  M(c_3nq \log \delta+2L\log n)+ NL\log n+ M_1 h(e_n)+ (N-M_1)\varepsilon n.
    \end{split}
\end{align}
Since $e_n=\frac{d_nU(n)}{V(n)}$ then,
\begin{equation}\label{eqh}
h(e_n)=h\Bigg(\frac{d_nU(n)}{V(n)}\Bigg) \leq h(d_n)+h(U(n))+h(V(n))\leq nC_1,
\end{equation}
for large values of $n$, and $C_1$ depends only on recurrences $U$ and $V$. Using \eqref{eqh} in \eqref{eq13}, we find 
\begin{align}\label{eq14}
\begin{split}
    \log  \Big(\prod_{w\in S}\prod_{j=1}^N &{|L_{w, j}(x_1(n),\ldots,x_{N}(n))|_{w}}\Big)\\
    &\leq  (C_2M_1+c_3Mq \log \delta)n+3NL\log n.
    \end{split}
\end{align}
Denote $C_3=C_2/(-\log \delta)$ and hence $C_3$ is a positive real number that depends only on $U$ and $V$. Now we impose conditions on the positive integers $q,h,k$ as follows: choose 
\[q>2C_3/c_3 \text{ and } k>3qD.\]
This gives 
\[c_3qk>2C_3k>\frac{3}{2}C_3(k+qD)>C_3(k+qD).\]
Note that the function $\binom{p+y}{p}$ is a polynomial in $y$ of degree $p$. Hence,
\begin{equation}\label{binom}
\binom{p+h}{p}c_3qk> C_3\binom{p+h+qD}{p}(k+qD)
\end{equation} 
for large $h$. Thus, for fixed $q, D, C_3, k, p$ with $q>2C_3/c_3$ and $k>3qD$, both sides of \eqref{binom} are polynomials in $h$ of degree $p$ with the leading coefficient on the left being larger than the leading coefficient on the right. Therefore, we may choose $h$ large enough so that \eqref{binom} is satisfied. This inequality implies that $c_3Mq>C_3 M_1$ for large $h$. Substituting the value of $C_3$, we obtain $C_2M_1<-c_3Mq \log \delta$. The inequality \eqref{binom} expresses the fact that in vector ${\bf x}(n)$, the coordinates involving $e_n$ are fewer in number.

Also in \eqref{eq14}, $N, L$ are fixed integers and for large values of $n$ we have $\log(n)/n$ tends to $0$. Thus, we can find a constant $C_4>0$, independent of $n$, such that
\begin{equation}\label{eq15}
 \log  \Big(\prod_{w\in S}\prod_{j=1}^N {|L_{w, j}(x_1(n),\ldots,x_{N}(n))|_{w}}\Big)<-C_4n,   
\end{equation}
for large $n \in \mathcal{M}$. 
\par
Since each coordinate of the point ${\bf x}(n)$ has exponential growth at most, we have $h({\bf x}(n)) \leq C_5n$ with $C_5>0$, independent of $n$. Hence \eqref{eq15} implies that
\[ \log  \Big(\prod_{w\in S}\prod_{j=1}^N {|L_{w, j}(x_1(n),\ldots,x_{N}(n))|_{w}}\Big)<-\frac{C_4}{C_5}h({\bf x}(n)).\]
By applying Theorem \ref{schli} with $\varepsilon=\frac{C_4}{C_5}$, we get a non-trivial linear relation of the form 
\begin{equation}\label{eq16}
 C_1'x_1(n)+\cdots+C_N'x_N(n)=0,   
\end{equation}
with $C_i' \in K$ for $i=1,\ldots, N$ not all zero, which is valid for infinitely many $n\in \mathcal{M}$. 
\par 
Recall that for $j\leq M_1$, $x_j(n)$ are of the form $n^l{\underline\gamma}^{n{\bf b}}e_n$ for a suitable vector $({\bf b}, l)$ depending on $j$ and for $M_1<j\leq N$, they are of the form $n^l\gamma^nd_n$ with $\gamma \in \Gamma$. We rewrite \eqref{eq16} as
\[C_1'x_1(n)+\cdots+C_{M_1}'x_{M_1}(n)=C_{M_1+1}'x_{M_1+1}(n)+\cdots+C_N'x_N(n).\]
Thus, we obtain a relation 
\[e_nC'(n)=B(n),\]
for an infinite subsequence of integers $n \in \mathcal{M}$, with linear recurrences $C'(n), B(n)$ with roots in $\Gamma$. Also, all the roots of $C'(n)$ lie in $\Gamma^*$ generated by $\beta_1,\ldots, \beta_p$. 
Observe that if $C_i'=0$ for all $i=1,\ldots, M_1$, then the recurrence $B(n)$ vanishes for an infinite sequence of integers. But by Theorem \ref{thmskm}, we get a contradiction. Hence $C_i'$ cannot be zero for all $i=1,\ldots, M_1$. Thus $C'(n)$ is a non-zero recurrence with roots in $\Gamma^*$.
\par 
Since $e_n=\frac{d_nU(n)}{V(n)}$, so
\[d_nU(n)C'(n)=B(n)V(n),\]
for infinitely many $n \in \mathcal{M}$, where all four recurrences $U, C', B, V$ have their roots in $\Gamma$. Again by Theorem \ref{thmskm}, this relation holds identically. By Lemma \ref{lem2.2}, we get a relation 
\[d_nuc'=bv'\]
in the ring $\C[X, T_1, \ldots, T_t, T_1^{-1}, \ldots, T_t^{-1}]$. Note that under the isomorphism in the Lemma \ref{lem2.2}, integer $d_n$ maps to itself, and we may notate $d_n=d_X$. By assumption, $u, v'$ are co-prime. So $v'$ must divide $c'$.
\par
Since $c'\in \C[X, T_1, \ldots, T_p, T_1^{-1}, \ldots, T_p^{-1}]$, it easily follows that, $v'=f\rho$, where $\rho $ is a product of powers of $T_1, \ldots, T_t$ and 
\[f\in \C[X, T_1, \ldots, T_p, T_1^{-1}, \ldots, T_p^{-1}].\]
This implies that all roots of $V(n)$ have the same $v$-adic absolute value, which gives a contradiction. This completes part (i) of Proposition \ref{prop}.

Now assume the case $m< n$. Let 
\[U(m)=\sum_{i=1}^r u_i(m)\alpha_i^m, \;\;\mbox{ and }\;\; V(n)=\sum_{i=1}^t v_i(n)\beta_i^n.\] We chose $S$ large enough so that roots and coefficients of both $U$ and $V$ are $S$-units. If necessary by dividing $U(m)$ by $\alpha_1^m$ and $V(n)$ by $\beta_1^n$ assume that $\alpha_1=1$ and $\beta_1=1$ without changing the set 
\begin{align*}
\Lambda:=\Big\{(m, n) \in \N^2:& m<n \text{ and } \exists  d_{m, n}\in \Z\backslash \{0\}\\
&\text{ such that}\;\; \dfrac{d_{m, n}U(m)}{V(n)} \in \mathcal{O}_{K, S}\Big\}.
\end{align*}

Suppose that there are infinitely many pairs $(m, n)\in \Lambda$ with $m<n$ and the conclusion of part (ii) does not hold. Since $d_{m, n}U(m)/V(n) \in \mathcal{O}_{K, S}$ for infinitely many pairs $(m, n)\in \Lambda$, then by the definition of group of $S$-integers we have 
\[\Bigg|\frac{d_{m, n}U(m)}{V(n)}\Bigg|_{\nu}\leq 1\] for all $\nu \notin S$, which implies $|d_{m, n}U(m)|_{\nu}\leq |V(n)|_{\nu}$. Hence, \begin{equation*}
\sum_{\nu \notin S}-\log^-\max\{|d_{m, n}U(m)|_{\nu}, |V(n)|_{\nu}\}=\sum_{\nu \notin S}-\log^-|V(n)|_{\nu}.
\end{equation*}
Since $d_{m, n}U$ and $V$ are co-prime and $\beta_1=1$, by Lemma \ref{V(n)}, we get
\begin{equation}\label{eqgcd}
\sum_{\nu \notin S}-\log^-\max\{|d_{m, n}U(m)|_{\nu}, |V(n)|_{\nu}\}=\sum_{\nu \notin S}-\log^-|V(n)|_{\nu}<\epsilon n,
\end{equation}
for all but finitely many $(m,n)\in \Lambda$.

Let $H_n \subseteq\mathbb{P}^{t-1}$ be the moving hyperplane defined by $v_1(n)x_1+\cdots+v_t(n)x_t(n)=0$. Consider the moving points 
\[\beta_n=[\beta_1^n:\cdots:\beta_t^n]:\N\rightarrow \mathbb{P}^{t-1}(K).\]
Since $V$ has more than one root, we can set $t\geq 2$.
For $\nu \notin S$,
\begin{align}\label{eqmoving}
\begin{split}
\lambda_{H_n,\nu}(\beta(n))&:=\log \Big (\frac{\Vert\beta(n) \Vert_{\nu} \Vert H_n\Vert_{\nu}}{|v_1(n)\beta_1^n+\cdots+v_t(n)\beta_t^n|_{\nu}}\Big)\\&=\log \Vert H_n\Vert_{\nu}-\log |V(n)|_{\nu}<\epsilon n
\end{split}
\end{align}
for all but finitely many $n$ satisfying \eqref{eqgcd}.

Since $\Gamma$ is torsion-free and $V(n)$ has more than one root, $\beta_i/\beta_j$ is not a root of unity for $i\neq j$. So, the growth of $h(\beta_i^n/\beta_j^n)$ is same as $h(\beta_1^n,\ldots, \beta_t^n)$. If case (ii) of Lemma \ref{lemmagcd} holds, then we get 
\[h(\beta_i^n/\beta_j^n)=o(h(\beta_1^n,\ldots, \beta_t^n)).\] But both have the same growth, which is a contradiction. So we apply case (i) of Lemma \ref{lemmagcd} for $\epsilon_0=\epsilon/|S|$, we get 
\[\lambda_{H_n,\nu}(\beta(n))<\epsilon_0n\] for infinitely many $n \in \Lambda$. Combining this for $\nu \in S$ with \eqref{eqmoving}, we get
\[h(H_n)+nh(\beta_1,\ldots, \beta_t)=\sum_{\nu \in M_K}\lambda_{H_n,\nu}(\beta(n))<2\epsilon n\]
for infinitely many $n \in \Lambda$.
This is impossible since $h(\beta_1,\ldots, \beta_t)>1$. This completes the proof of part (ii) of Proposition \ref{prop}.
\end{proof}
\subsection{Proof of Theorem \ref{thm1}}
We first prove part (i). Assume that $d_{n}U(n)$ and $V(n)$ are co-prime. If not, we simplify the fraction $d_{n} U(n)/V(n)$ for the condition of co-primality. If $V(n)$ has only one root, then $V(n)=P(n)\beta^n$, where $P$ is a polynomial, and $\beta$ is the only one root of $V$. Clearly, $P(n)\beta^n/P(n)$ is a linear recurrence.
Now consider 
\[\frac{d_{n}P(n) U(n)}{P(n)\beta^n}=\frac{d_{n} U(n)}{\beta^n}.\] Substituting the expression for $U(n)$, we get
\[\frac{\sum_{i=1}^r d_{n}u_i(n)\alpha_i^n}{\beta^n}=\sum_{i=1}^r u_i'(n)\Big(\frac{\alpha_i}{\beta}\Big)^n,\]
which is a linear recurrence with $u_i'(n)=d_nu_i(n)$ and roots $\alpha_i/\beta$.

Similarly, for part (ii), we assume that $d_{m, n}U(m)$ and $V(n)$ are co-prime. If $V(n)$ has only one root, then as in part(i), $V(n)/P(n)$ is a linear recurrence. Also, \[\frac{d_{m, n}P(n) U(m)}{P(n)\beta^n}=\frac{d_{m, n} U(m)}{\beta^n}.\] Substituting the expression for $U(m)$, we get
\[\frac{\sum_{i=1}^r d_{m, n}u_i(m)\alpha_i^m}{\beta^n}=\sum_{i=1}^r u_i'(m, n)\alpha_i^m(\beta^{-1})^n,\]
which is a multi-recurrence with $u_i'(n)=d_{m, n}u_i(n)$ and bases $\alpha_i$ and $\beta^{-1}$. So, if $V$ has only one root, then Theorem \ref{thm1} holds. 

Now suppose that $V$ has more than one root. We apply Proposition \ref{prop}, by choosing appropriate $K$ and $S$. Using part (i) of Proposition \ref{prop}, we get that the set of $n\in \N$ is finite and using part (ii) of Proposition \ref{prop} we find that the set  $\mathcal{M}$ of pairs $(m, n)\in \N^2$ is finite, which contradicts the assumptions of Theorem \ref{thm1}. This completes the proof of Theorem \ref{thm1}.\qed

Next, to deduce the general case Corollary \ref{cor1} from Theorem \ref{thm1}, we use specialisation arguments developed by van der Poorten and Rumely (\cite{Rumely}).  The following lemma we adopt to prove this specialisation part. 
\begin{lemma}[\cite{Rumely}]\label{lem2}
Let $\mathcal{O}$ denote a finitely generated subring of $\C$. Let $\delta \in \mathcal{O}$ be non-zero and let $\Gamma$ be a finitely generated torsion-free subgroup of $\mathcal{O^{\times}}$. Then there exists a ring homomorphism $\phi: \mathcal{O}\rightarrow  \bar{\Q}$ such that $\phi(\delta)\neq0$ and the restriction of $\phi $ to $\Gamma$ is injective.
\end{lemma}

\subsection{Proof of Corollary \ref{cor1}} First, we prove part (ii). Assume that the set $\mathcal{M}$ is the set of pairs $(m, n)$ satisfying the conditions of Corollary \ref{cor1}. Clearly $\mathcal{M}$ is an infinite set.  We let $\mathcal{O}$ denote the ring generated over $\mathcal{R}$ by all coefficients, roots and their respective reciprocals of both $U$ and $V$. We denote by $\Gamma$ the group generated by the roots of $U$ and $V$. By assumption, it is a torsion-free group, and we denote the generators of $\Gamma$ by $\gamma_1, \ldots, \gamma_t$. 
\par 
Using Lemma \ref{lem2.2}, we may define an isomorphism which associates the variable $X$ to the function $n \rightarrow n$ and the variables $T_i$ to the function $n\rightarrow \gamma_i^n$. Denote polynomials \[d_{m, n}u, v \in \C[X, T_1, \ldots, T_t, T_1^{-1}, \ldots, T_t^{-1}]\] corresponding to recurrences $d_{m, n}U(m), V(n)$ respectively. Since $d_{m, n}$ are integers, under the isomorphism, they are fixed.
Note that the units in the ring $\C[X, T_1, \ldots, T_t, T_1^{-1}, \ldots, T_t^{-1}] $  are precisely the terms $qT_1^{a_1}\cdots T_t^{a_t}$ with $q \in \C$ and $a_i \in \Z$. Assume that $d_{m, n}u, v$ are co-prime and $v$ is not a unit multiple of an element of $\C[X]$. Suppose that both $d_{m, n}u,v$ lie in the ring $\C[X, T_1, \ldots, T_t]$ by multiplying by a suitable unit such that both are co-prime and $v$ has more than one term as a polynomial in $T_1, \ldots, T_t$. Hence, there exists a variable, say $T_1$, appearing in the terms of $v$ with at least two different degrees. Also, suppose that $d_{m, n}u,v$ lies in the ring $\mathcal{O}[X, T_1, \ldots, T_t]$.
 \par
 Now consider the resultant $\omega(X, T_2, \ldots, T_t)$ of $d_{m, n}u, v$ with respect to $T_1$. As $d_{m, n}u$ and $v$ are co-prime, it is clear that the resultant is non-zero and has coefficients in $\mathcal{O}$. To apply Lemma \ref{lem2}, we let $\delta$ to be the product of the non-zero coefficients of $\omega$ and of $d_{m, n}u,v$. Let $\phi$ be as in Lemma \ref{lem2}. It is easy to verify that $\phi(\gamma_i)$ are multiplicatively independent. If not, then we get a contradiction to the fact that $\gamma_i$ are multiplicatively independent for $i=1,\ldots, t$.
\par
Consider the specialisations $\phi(d_{m, n}u)$ and $\phi (v)$. These are polynomials in the ring $\bar{\Q}[X, T_1, \ldots, T_t]$ and are co-prime with respect to $T_1$. Here we used the fact that the resultant of $d_{m, n}u, v$ is equal to the resultant of $\phi(d_{m, n}u), \phi(v)$, and are non-zero. By the initial choice of $T_1$, and the non-vanishing of $\phi$ on the coefficients, we get $\phi(v)$ also contains at least two terms. We may write $\phi(d_{m, n}u)=du_1 \text{ and  } \phi(v)=dv_1,$
where $d, u_1, v_1 \in \bar{\Q}[X, T_1, \ldots, T_t]$ and $u_1, v_1$ are co-prime. It is obvious that $d$ is independent of $T_1$ and $v_1$ contains at least two terms with respect to $T_1$. 
\par 
Again using Lemma \ref{lem2.2}, we associate $u_1, v_1$ to co-prime linear recurrences $\tilde{U}, \tilde{V}$ with algebraic roots and coefficients. Here, we associate the variable $X$ to the function $n \rightarrow n$ and the variables $T_i$ to the function $n\rightarrow \phi(\gamma_i)^n$. As $\phi$ is injective on $\Gamma$ and $v_1$ contains at least two terms, $\tilde{V}$ has at least two roots. Also, the roots of these recurrences generate a torsion-free subgroup. These recurrences are non-degenerate, so by  Theorem \ref{thmskm}, $\tilde{V}=0$ only for finitely many $n$. We may disregard this finite set. 
\par 
Then 
\[\frac{d_{m, n} \tilde{U}(m)}{\tilde{V}(n)}=\phi\Bigg(\frac{d_{m, n}U(m)}{V(n)}\Bigg) \in \phi(\mathcal{R)},\]
for $(m, n)\in \mathcal{M}$. Since $\phi(\mathcal{R})$ is a finitely generated subring of a number field, then by assumption $\mathcal{M}$ is infinite and $\tilde{V}$ has at least two roots. This is a contradiction to Proposition \ref{prop}. To obtain part (i), we do the same steps by putting $m=n$ everywhere. This completes the proof of corollary \ref{cor1}.
\qed

\subsection{Proof of Corollary \ref{cor2}} 
We prove part (ii), and the proof of part (i) follows in the same way by putting $m=n$. Assume that for all but finitely many $(m, n)\in \mathbb{N}^2$ there exist non-zero positive integers $d_{m, n}$ such that $\log |d_{m, n}|=o(n)$,  and \[d_{m,n}U(m)/V(n) \in \mathcal R.\]
Now we partition $\N$ into a finite number of suitable arithmetic progressions and consider the restrictions of the functions involved to each progression separately. Observe that here, the multiplicative group generated by the roots of $U$ and $V$ need not be torsion-free. Let $q$ be the order of the torsion in the multiplicative group $\Gamma$ generated by the roots of $U, V$ together. For $r_1, r_2=0,1, \ldots, q-1$, let $m=qm_1+r_1$ and $n=qn_1+r_2$. Then the recurrences $U(qm_1+r_1)$ and $V(qn_1+r_2)$ have roots among the $q$th powers of the roots of $U, V$ and they are in the torsion-free group $\Gamma^q$. Now we apply Theorem \ref{thm1} in each of these partitions. So we get a non-zero polynomial $P(X) \in \mathbb{C}[X]$ and positive integers $q, r_1, r_2$ such that $d_{qm_1+r_1, qn_1+r_2}P(qn_1+r_2)U(qm_1+r_1)/V(qn_1+r_2)$ is a multi-recurrence and $V(qn_1+r_2)/P(qn_1+r_2)$ is a linear recurrence. Now in the proof of part (i), when $m=n$, $d_{qn_1+r_2}P(qn_1+r_2)U(qn_1+r_2)/V(qn_1+r_2)$ becomes a linear recurrence as in the proof of Theorem \ref{thm1}. Hence, the proof of Corollary \ref{cor2}.

\section*{Acknowledgement} 
The first author's research is supported by a UGC fellowship (Ref No. 221610077314).


\begin{thebibliography}{9999}
		

\bibitem{zeros}
F. Amoroso and E. Viada, \emph{On the zeros of linear recurrence sequences}, Acta Arith. {\bf 147}(4) (2011), 387-396.

\bibitem{divisibility}
R. André-Jeannin, \emph{Divisibility of generalized Fibonacci and Lucas numbers by their subscripts}, Fibonacci Quart. {\bf 29}(4) (1991), 364-366.

\bibitem{corv1998} 
P. Corvaja and U.  Zannier, {\em Diophantine equations with power sums and universal Hilbert sets}, Indag. Math. (N.S.) {\bf 9} (3) (1998), 317-332. 

\bibitem{corv}
P. Corvaja and U.  Zannier, \emph{Finiteness of integral values for the ratio of two linear recurrences}, Invent. Math. {\bf149} (2002), 431-451.

\bibitem{evertse}
G. Everest, A. J. van Der Poorten, I. E. Shparlinski and T.  Ward, \emph{Recurrence sequences}, Math. Surveys Monogr., Vol. 104, AMS, Providence, RI (2003).
\bibitem{numbers n}
J. J. A. González, F. Luca, C.  Pomerance and  I. E. Shparlinski, \emph{On numbers n dividing the nth term of a linear recurrence}, Proc. Edinb. Math. Soc. {\bf 55}(2) (2012), 271-289.

\bibitem{Granville2022}
A. Granville, \emph{Classifying linear division sequences},  Amer. J. Math. (to appear), arXiv:2206.11823.

\bibitem{gcd}
N. Grieve and J.T.-Y. Wang, \emph{Greatest common divisors with moving targets and consequences for linear recurrence sequences}, Trans. Amer. Math. Soc. {\bf 373}(11) (2020), 8095-8126.

\bibitem{lech}
C. Lech, \emph{A note on recurring series}, Ark. Mat. {\bf 2}(5) (1953), 417-421.

\bibitem{levin}
A. Levin, \emph{Greatest common divisors and Vojta’s conjecture for blowups of algebraic tori}, Invent. Math. {\bf 215}(2) (2019), 493-533.

\bibitem{Mahler1}
K. Mahler, \emph{Eine arithmetische Eigenschaft der Taylor-Koeffizienten  rationaler Funktionen}, Proc. K. Ned. Akad. Wet. {\bf 38} (1935), 50-60.

\bibitem{Mahler2}
K. Mahler, \emph{On the Taylor coefficients of rational functions}, Proc. Camb. Philos. Soc. {\bf 52} (1956), 39-48.

\bibitem{Mahler3}
K. Mahler, \emph{Addendum to the paper ``On the Taylor coefficients of rational functions"}, Proc. Camb. Philos. Soc. {\bf 53} (1957), 544.

\bibitem{Pourchet}
Y. Pourchet, \emph{Solution du probleme arithm\'etique du quotient de Hadamard de deux fractions rationnelles}, C.R. Acad. Sci. Paris {\bf 288} (1979), 1055-1057.

\bibitem{Ru1997}
M. Ru and P. Vojta, \emph{Schmidt's subspace theorem with moving targets}, Invent. Math. {\bf 127}(1) (1997), 51-65.

\bibitem{kthroot}
R. Rumely and A. J. van der Poorten, \emph{A note on the Hadamard kth root of a rational function}, J. Aust. Math. Soc. {\bf 43}(3) (1987), 314-327.

\bibitem{Rumely}
R. Rumely, \emph{Notes on van der Poorten's proof of the Hadamard Quotient Theorem, S\'em. de Th\'eorie des Nombres, Paris 1986–87}, Progr. Math. {\bf 75} (1988), 349-409.

\bibitem{sanna}
C. Sanna, \emph{Distribution of integral values for the ratio of two linear recurrences}, J. Number Theory {\bf180} (2017), 195-207. 

\bibitem{Subspace}
W. M. Schmidt, \emph{Diophantine approximations and Diophantine equations}, Springer (2006).

\bibitem{shorey}
T. N. Shorey and R. Tijdeman, \emph{Exponential Diophantine Equations}, Vol. 87, Cambridge University Press (1986).

\bibitem{skolem}
T. Skolem, \emph{Ein Verfahren zur Behandlung gewisser exponentialer Gleichungen und diophantischer Gleichungen}, Comptes Rendus Congr. Math. Scand., Stockholm (1934), 163-188.

\bibitem{vdP2}
A. J. van der Poorten, \emph{Solution de la conjecture de Pisot sur le quotient de Hadamard de deux fractions rationnelles}, C. R Acad. Sci. Paris {\bf 306}(97) (1988).

\bibitem{VdP} 
A. J. van der Poorten, \emph{Some facts that should be better known, especially about rational functions in Number Theory and Applications (Banff, AB, 1988)}, {\bf265} (1989), 497-528.

\bibitem{Zhen}
Z. Xiao, \emph{Greatest common divisors for polynomials in almost units and applications to linear recurrence sequences}, Math. Z. {\bf 306}(4) (2024), article number 61.

\bibitem{zannier}
U. Zannier, {\em Diophantine equations with linear recurrences An overview of some recent progress}, J. Th\'eor. Nombres Bordx. {\bf 17}(1) (2005), 423-435.

\end{thebibliography}
\end{document}